\documentclass[12pt]{article}%
\usepackage{amsmath}
\usepackage{amsfonts}
\usepackage{mitpress}
\usepackage{color}
\usepackage{amssymb}
\usepackage{graphicx}%
\providecommand{\U}[1]{\protect\rule{.1in}{.1in}}
\newtheorem{theorem}{Theorem}

\newtheorem{corollary}[theorem]{Corollary}

\newtheorem{example}[theorem]{Example}

\newtheorem{lemma}[theorem]{Lemma}

\newtheorem{proposition}[theorem]{Proposition}
\newtheorem{remark}[theorem]{Remark}

\newenvironment{proof}[1][Proof]{\noindent\textbf{#1.} }{\ \rule{0.5em}{0.5em}}
\newdimen\dummy
\dummy=\oddsidemargin
\begin{document}

\title{Stretched random walks and the behaviour of their summands.}
\author{Michel Broniatowski, Zhangsheng Cao\\LSTA, Universit\'{e} Paris 6}
\maketitle

\begin{abstract}
This paper explores the joint behaviour of the summands of a random walk when
their mean value goes to infinity as its length increases.\ It is proved that
all the summands must share the same value, which extends previous results in
the context of large exceedances of finite sums of i.i.d. random variables.
Some consequences are drawn pertaining to the local behaviour of a random walk
conditioned on a large deviation constraint on its end value.\ It is shown
that the sample paths exhibit local oblic segments with increasing size and
slope as the length of the random walk increases.

\end{abstract}

\bigskip


\section{Introduction}

\subsection{Context and scope}

This paper considers the following question: Let $X,X_{1},..,X_{n}$ denote
real valued independent random variables (r.v's) distributed as $X$ and let
$S_{1}^{n}:=X_{1}+..+X_{n}.$ We assume that $X$ is unbounded upwards. Let
$a_{n}$ be some positive sequence satisfying%
\begin{equation}
\lim_{n\rightarrow\infty}a_{n}=+\infty. \label{LimA_n=infty}%
\end{equation}
Assuming that
\begin{equation}
C:=\left(  S_{1}^{n}/n>a_{n}\right)  \label{EventC}%
\end{equation}
holds, what can be inferred on the r.v's $X_{i}$'s as $n$ goes to infinity?

Let $\varepsilon_{n}$ denote a positive sequence and let
\begin{equation}
I:=\cap_{i=1}^{n}\left(  X_{i}\in(a_{n}-\epsilon_{n},a_{n}+\epsilon
_{n})\right)  . \label{EventI}%
\end{equation}

We consider cases when
\begin{equation}
\lim_{n\rightarrow\infty}P\left(  \left.  I\right\vert C\right)  =1.
\label{QUESTION}%
\end{equation}
The relation between the various parameters in this problem is of interest and
opens a variety of questions. For which distributions\ $P_{X}$ pertaining to
$X$ is such a result valid? Which is the acceptable growth of the sequence
$a_{n}$ and the possible behaviours of the sequence $\varepsilon_{n}$ such
that
\begin{equation}
\epsilon_{n}=o\left(  a_{n}\right)  \label{epsilon_n}%
\end{equation}
and is it possible to achieve
\begin{equation}
\lim_{n\rightarrow\infty}\epsilon_{n}=0 \label{EpsTendsToZero}%
\end{equation}
under a large class of choices for $P_{X}$?

In the case when the r.v. $X$ has light tails conditional limit theorems
exploring the behavior of the summands of a random walk given its sum have
been developped extensively in the range of a large deviation conditioning
event, namely similar as defined by $C$ with fixed $a_{n}$, hence
lower-bounding $S_{n}/n$ independently on $n$; $\ $the papers \cite{CS1984},
or \cite{DIFR1988} together with their extension in \cite{DEZE1996} explore
the asymptotic properties of a relatively small number of summands; the main
result in these papers, named as Gibbs conditional principle, lies in the fact
that under \ such $C$, the $X_{i}$'s are asymptotically i.i.d. with
distribution $\Pi^{a}$ defined through $d\Pi^{a}(x):=\left(  E\left(  \exp
tX\right)  \right)  ^{-1}\exp(tx)dP_{X}(x)$ where $t$ satisfies $E\left(
X\exp tX\right)  \left(  E\left(  \exp tX\right)  \right)  ^{-1}=a$ ; in this
range (\ref{EpsTendsToZero}) does not hold.\ The joint distribution of
$X_{1},..,X_{k_{n}}$ given $C$ (with fixed $a_{n}$) for large $k_{n}$ (close
to $n$) is considered in \cite{BRCA2012} .

Extended large deviations results for $a_{n}\rightarrow\infty$ have been
considered in \cite{BR1987}, \cite{BRMA1994}, in relation with versions of the
Erd\"{o}s-R\'{e}nyi law of large numbers for the small increments of a random
walk, and \cite{JUNA2004}.\ 

The case when $X$ is heavy tailed is considered in \cite{AM2011} where the
authors consider the support of the distribution of the whole sample
$X_{1},..,X_{n}$ $\ $\ when $C$ holds for fixed $a_{n}$ .

A closely related problem has been handled by statisticians in various
contexts, exploring the number of sample observations which push a given
statistics far away from its expectation, for \textit{fixed} $n$. Although
similar in phrasing as the so-called "breakdown point" paradigm of robust
analysis , the frame of this question is quite different from the robustness
point of view, since all the observations are supposed to be sampled under the
distribution $P_{X}$, hence without any reference to outliers or
misspecification.\ The question may therefore be stated as: how many sample
points should be large making a given statistics large? This combines both the
asymptotic behavior of the statistics (as a function defined on $\mathbb{R}%
^{n}$) and the tail properties of $P_{X}$. In the case when the statistics is
$S_{1}^{n}/n$ and $X$ has subexponential upper tail, it is well known that, denoting%

\[
C_{a}:=\left(  S_{1}^{n}/n>a\right)
\]
only one large value of the $X_{i}$'s generates $C_{a}$ for $a\rightarrow
\infty$; clearly $S_{n}/n$ is not a loyal statistics under this sampling. This
result turns back to Darling (1952). For light tails, under $C_{a}$, all
sampled values should exceed $a$ (indeed they should be closer and closer to
$a$ as $a\rightarrow\infty$), so that $S_{n}/n$ is faithfull in allegeance
with respect to the sample. In this case, denoting%

\[
I_{a}:=\cap_{i=1}^{n}\left(  X_{i}>a\right)
\]
it holds
\begin{equation}
\lim_{a\rightarrow\infty}P\left(  \left.  I_{a}\right\vert C_{a}\right)  =1
\label{democracy}%
\end{equation}
Intermediate cases exist, leading to partial loyalty for a given statistics
under a given sampling scheme. See \cite{BRFU1995}, \cite{BEBRTEVY1995}, and
\cite{BABR2004} where more general statistics than $S_{n}/n$ are
considered.\ and $a\rightarrow\infty.$ According to the tail behavior of the
distribution of $X$ the situation may take quite different features.

Related questions have also been considered in the realm of statistical
physics.\ In \cite{FRSO1997} the property (\ref{democracy}) is stated in an
improved form, namely stating that when the $X_{i}$'s are i.i.d. with Weibull
density with shape index larger than 2 then the conditional density of
$(X_{1},..,X_{n})$ given $\left(  S_{1}^{n}/n=a\right)  $ concentrates at
$\left(  a,..a\right)  $ as $a\rightarrow\infty$, which in the authors' words
means that the $X_{i}$'s are \textit{democratically localized}. Applications
of this concept in fragmentation processes, in some form of anomalous
relaxation of glasses and in the study of turbulence flows are discussed.

We now come to a consequence of the present results considering the local
behaviour of a random walk conditioned on its end value. Let $S_{i}^{j}%
:=X_{i}+..+X_{j}$ with $1\leq i\leq j\leq n$ $\ $\ and $k=k_{n}$ denote an
integer valued sequence such that%
\[
k_{n}\leq n
\]
and
\[
\lim_{n\rightarrow\infty}k_{n}=\infty.
\]
Let further
\[
\Delta_{j,n}:=S_{j+1}^{j+k}/k
\]
denote the local slope of the random walk on the interval $\left[
j+1,j+k\right]  $ where $1\leq j\leq n-k.$ The limit behaviour of $\max_{1\leq
j\leq n-k}\Delta_{j,n}$ has been considered extensively in various cases,
according to the order of magnitude of $k$. The case $k=C\log n$ for positive
constant $C$ defines the so-called Erd\"{o}s-R\'{e}nyi law of large numbers;
see \cite{ERRE1970}. In the present case we consider random walks conditioned
upon their end value, namely assuming that
\[
S_{1}^{n}>na
\]
for fixed $a>EX.$ We will prove that as $n\rightarrow\infty$ the path defined
by this random walk exhibits anomalous local behavior that can be captured
through the \textit{extended democratic localization principle} stated in our
results. Indeed there exist segments of length $k_{n}$ on which the slope
$\Delta_{j,n}$ tends to infinity with a rate which can be made precise.
Simulations are proposed in order to enlight this phenomenon. Obviously, when
$a$ is not fixed but goes to infinity with $n$ then the \textit{extended
democratic localization principle }applies to the whole sample path of the
random walk, and its trajectory is nearly a stright line from the origin up to
its extremity. When conditioning in the range of the large deviation only,
this property holds locally.

This paper is organized as follows. Section 2 states the notation and
hypotheses. Section 3 states the results in two cases; the first one pertains
to the case when $X$ has a log-concave density and the second case is a
generalizetion of the former.\ Examples are, provided. Section 4 presents a
short account on the local behaviour of random paths from conditioned random
walk, with some simulation. The proofs of the results are rather long and
technical; they have been postponed to the Appendix.

\section{Notation and hypotheses}

The $n$ real valued random variables $X_{1},...,X_{n}$. are independent copies
of a r.v. $X$ with density $p$ whose support is $\mathbb{R}^{+}.$ As seen by
the very nature of the problem handled in this paper, this assumption puts no
restriction to the results. We write
\[
p(x):=\exp-h(x)
\]
for positive functions $h$ which are defined and denoted according to the
context. For $\mathbf{x}\in\mathbb{R}^{n}$ define
\[
I_{h}\left(  \mathbf{x}\right)  :=\sum_{1\leq i\leq n}h(x_{i}),
\]
and for $A$ a Borel set in $\mathbb{R}^{n}$ denote
\[
I_{h}(A)=\inf_{\left(  \mathbf{x}\right)  \in A}I_{h}\left(  \mathbf{x}%
\right)  .
\]
\bigskip

Two cases will be considered: in the first one $h$ is assumed to be a convex
function, and in the second case $h$ will be the sum of a convex function and
a "smaller" function $h$ in such a way that we will also handle non
log-concave densities.(although not too far from them). Hence we do not
consider heavy tailed r.v. $X.$

For positive $r$ define
\[
S(r)=\left\{  \mathbf{x}:=\left(  x_{1},..,x_{n}\right)  :\sum_{1\leq i\leq
n}h(x_{i})\leq r\right\}  .
\]

\section{Very Large Deviation for Exponential Density Functions associated to
Convex Functions}

\begin{lemma}
\label{Lemma01} Let $g$ be a positive convex differentiable function defined
on $\mathbb{R}_{+}$ . Assume that $g$ is strictly increasing on some interval
$[X,\infty).$ Let (\ref{LimA_n=infty}) hold. Then
\[
I_{g}(I^{c}\cap C)=\min\big(F_{g_{1}}(a_{n},\epsilon_{n}),F_{g_{2}}%
(a_{n},\epsilon_{n})\big),
\]
where
\[
F_{g_{1}}(a_{n},\epsilon_{n})=g(a_{n}+\epsilon_{n})+(n-1)g\left(  a_{n}%
-\frac{1}{n-1}\epsilon_{n}\right)  ,
\]
and
\[
F_{g_{2}}(a_{n},\epsilon_{n})=g(a_{n}-\epsilon_{n})+(n-1)g\left(  a_{n}%
+\frac{1}{n-1}\epsilon_{n}\right)  .
\]

\end{lemma}

\begin{theorem}
\label{theoremConvex}Let $X_{1},...,X_{n}$ be i.i.d. copies of a r.v. $X$ with
density $p(x)=c\exp\left(  {-g(x)}\right)  $, where $g(x)$ is a positive
convex function on $\mathbb{R}^{+}.$ Assume that $g$ is increasing on some
interval $[X,\infty)$ and satisfies
\[
\lim_{x\rightarrow\infty}g(x)/x=\infty.
\]
Let $a_{n}$ satisfy
\[
\lim\inf_{n\rightarrow\infty}\frac{\log a_{n}}{\log n}n>0
\]
and that for some positive sequence $\epsilon_{n}$
\begin{align}
\lim_{n\rightarrow\infty}\frac{n\log g\left(  a_{n}+\epsilon_{n}\right)
}{H(a_{n},\epsilon_{n})}  &  =0,\label{condition32}\\
\lim_{n\rightarrow\infty}\frac{nG(a_{n})}{H(a_{n},\epsilon_{n})}  &  =0,
\label{condition33}%
\end{align}
where
\begin{align*}
&  H(a_{n},\epsilon_{n})=\min\left(  F_{g_{1}}(a_{n},\epsilon_{n}),F_{g_{2}%
}(a_{n},\epsilon_{n})\right)  -ng(a_{n}),\\
&  G(a_{n})=g(a_{n}+\frac{1}{g(a_{n})})-g(a_{n}),
\end{align*}
where $F_{g_{1}}(a_{n},\epsilon_{n})$ and $F_{g_{2}}(a_{n},\epsilon_{n})$ are
defined as in Lemma $\ref{Lemma01}$.Then as $n\rightarrow\infty$ it holds
$P(I|C)\rightarrow1$.
\end{theorem}

\begin{example}
\label{1ex01} Let $g(x):=x^{\beta}$. For power functions,through Taylor
expansion it holds
\[
g\left(  a_{n}+\frac{1}{g(a_{n})}\right)  -g(a_{n})=\frac{\beta}{a_{n}%
}+o\left(  \frac{1}{a_{n}}\right)  =o\left(  \log g(a_{n})\right)
\]
hence condition $(\ref{condition33})$ holds as a consequence of
$(\ref{condition32})$. If we assume that $\epsilon_{n}=o(a_{n})$, by Taylor
expansion we obtain
\[
\min\big(F_{g_{1}}(a_{n},\epsilon_{n}),F_{g_{2}}(a_{n},\epsilon_{n}%
)\big)=na_{n}^{\beta}+C_{\beta}^{2}\frac{n}{n-1}a_{n}^{\beta-2}\epsilon
_{n}^{2}+o(a_{n}^{\beta-2}\epsilon_{n}^{2}).
\]
Condition $(\ref{condition32})$ then becomes
\[
\lim_{n\rightarrow\infty}\frac{n\log a_{n}}{a_{n}^{\beta-2}\epsilon_{n}^{2}%
}=0.
\]

\textbf{Case 1:} $1<\beta\leq2$.

To make $(\ref{condition33})$ hold, we need $\epsilon_{n}$ be large enough,
specifically,
\[
a_{n}^{1-\frac{\beta}{2}}\sqrt{\log a_{n}}=o\left(  \epsilon_{n}\right)
=o\left(  a_{n}\right)
\]
which shows that $\epsilon_{n}\rightarrow\infty.$

\textbf{Case 2:} $\beta>2$.

In this case, if we take $n=a_{n}^{\alpha}$ with $0<\alpha<\beta-2$, then
condition $(\ref{condition33})$ holds for arbitrary sequences $\epsilon_{n}$
bounded by below away from $0.$ The sequence $\epsilon_{n}$ may also tend to
$0;$ indeed with $\epsilon_{n}=O(1/\log a_{n})$, condition $(\ref{condition33}%
)$\ holds. Also setting $a_{n}:=n^{\alpha}$ for $\alpha>0$ there exist
sequences $\epsilon_{n}$ which tend to $0$ such that the conclusion in Theorem
\ref{theoremConvex} holds.

\begin{example}
\label{1ex02} \textbf{Let } $g(x):=e^{x}$. Through Taylor expansion
\[
g\left(  a_{n}+\frac{1}{g(a_{n})}\right)  -g(a_{n})=1+o\left(  \frac{1}{a_{n}%
}\right)  =o\left(  \log g(a_{n})\right)  =o\left(  a_{n}\right)  ,
\]
and if $\epsilon_{n}\rightarrow0$, it holds
\[
\min\left(  F_{g_{1}}(a_{n},\epsilon_{n}),F_{g_{2}}(a_{n},\epsilon
_{n})\right)  =ne^{a_{n}}+\frac{1}{2}\frac{n}{n-1}e^{a_{n}}\epsilon_{n}%
^{2}+o(e^{a_{n}}\epsilon_{n}^{2}).
\]
Hence condition $(\ref{condition33})$ follows from condition
$(\ref{condition32});$ furthermore condition $(\ref{condition32})$ follows
from
\[
\lim_{n\rightarrow\infty}\frac{na_{n}}{e^{a_{n}}\epsilon_{n}^{2}}=0
\]
if we set $a_{n}:=n^{\alpha}$ where $\alpha>0$ then condition
$(\ref{condition33})$ holds, and $\epsilon_{n}$ is rapidly decreasing to $0;$
indeed we may choose $\epsilon_{n}=o(\exp(-{a_{n}}/{4}))$.
\end{example}
\end{example}

\begin{corollary}
\label{order of s} Let $X_{1},..,X_{n\text{ }}$be independent r.v's with
common Weibull density with shape parameter $k$ and scale parameter $1,$
\[
p(x)=%
\begin{cases}
kx^{k-1}e^{-x^{k}}\quad & when\;x>0\\
0 & otherwise,
\end{cases}
\]
where $k>2$. Let
\[
a_{n}=n^{\frac{1}{\alpha}},
\]
for some $0<\alpha<k-2$ and let $\epsilon_{n}$ be a positive sequence such
that
\[
\lim_{n\rightarrow\infty}\frac{n\log a_{n}}{a_{n}^{k-2}\epsilon_{n}^{2}}=0.
\]
Then%
\[
\lim_{n\rightarrow\infty}P(I|C)=0.
\]
.
\end{corollary}

Proof: Set $g(x)=x^{k}-(k-1)\log x$, which is a convex function for $k>2$.
Also when $x\rightarrow\infty$, $g^{\prime}(x)$ and $g^{\prime\prime}(x)$ are
both infinitely small with respect to $g(x)$ as $x\rightarrow\infty.$

Both conditions $(\ref{condition32})$ and $(\ref{condition33})$ in Theorem
$\ref{theoremConvex}$ are satisfied. As regards to condition
$(\ref{condition33})$, notice firstly that, under the Weibull density by
Taylor expansion
\[
g(a_{n}+\epsilon_{n})=g(a_{n})+g^{\prime}(a_{n})\epsilon_{n}+\frac{1}%
{2}g^{\prime\prime}(a_{n})\epsilon_{n}^{2}+o\left(  g^{\prime\prime}%
(a_{n})\epsilon_{n}^{2}\right)  .
\]
Hence it holds
\[
\log g\left(  a_{n}+\epsilon_{n}\right)  \leq\log\left(  3g(a_{n})\right)
\leq\log\left(  3a_{n}^{k}\right)  =\log3+k\log a_{n}.
\]
Using Taylor expansion in $g(a_{n}+\epsilon_{n})$ and $g\left(  a_{n}%
-\frac{\epsilon_{n}}{n-1}\right)  $, it holds
\begin{align*}
F_{g_{1}}(a_{n},\epsilon_{n}) &  -ng(a_{n})=g(a_{n}+\epsilon_{n}%
)+(n-1)g\left(  a_{n}-\frac{\epsilon_{n}}{n-1}\right)  -ng(a_{n})\\
&  =\left(  g(a_{n})+g^{\prime}(a_{n})\epsilon_{n}+\frac{1}{2}g^{\prime\prime
}(a_{n})\epsilon_{n}^{2}+o\left(  g^{\prime\prime}(a_{n})\epsilon_{n}%
^{2}\right)  \right)  \\
&  \quad+\left(  (n-1)g(a_{n})-g^{\prime}(a_{n})\epsilon_{n}+\frac{1}%
{2}g^{\prime\prime}(a_{n})\frac{\epsilon_{n}^{2}}{n-1}+o\left(  g^{\prime
\prime}(a_{n})\epsilon_{n}^{2}\right)  \right)  -ng(a_{n})\\
&  \geq\frac{1}{2}g^{\prime\prime}(a_{n})\epsilon_{n}^{2}+o\left(
g^{\prime\prime}(a_{n})\epsilon_{n}^{2}\right)  =\frac{k(k-1)}{2}a_{n}%
^{k-2}\epsilon_{n}^{2}+o\left(  a_{n}^{k-2}\epsilon_{n}^{2}\right)  .
\end{align*}
In the same way, it holds when $a_{n}\rightarrow\infty$
\[
F_{g_{2}}(a_{n},\epsilon_{n})-ng(a_{n})\geq\frac{k(k-1)}{2}a_{n}^{k-2}%
\epsilon_{n}^{2}+o\left(  a_{n}^{k-2}\epsilon_{n}^{2}\right)  .
\]
Thus we have
\[
H(a_{n},\epsilon_{n})\geq\frac{k(k-1)}{2}a_{n}^{k-2}\epsilon_{n}^{2}+o\left(
a_{n}^{k-2}\epsilon_{n}^{2}\right)  .
\]
Hence, when $n\rightarrow\infty$, with $(\ref{convexCorralory01}%
),(\ref{convexCorralory02})$, the condition $(\ref{condition32})$ of Theorem
$(\ref{theoremConvex})$ becomes
\begin{align*}
\frac{n\log g\left(  a_{n}+\epsilon_{n}\right)  }{H(a_{n},\epsilon_{n})} &
\leq\frac{n\log3+kn\log a_{n}}{\frac{k(k-1)}{2}a_{n}^{k-2}\epsilon_{n}%
^{2}+o\left(  a_{n}^{k-2}\epsilon_{n}^{2}\right)  }\\
&  \leq\frac{2kn\log a_{n}}{\frac{k(k-1)}{4}a_{n}^{k-2}\epsilon_{n}^{2}}%
=\frac{8}{k-1}\frac{n\log a_{n}}{a_{n}^{k-2}\epsilon_{n}^{2}}\longrightarrow0.
\end{align*}

The last step holds from condition $(\ref{X order's condition 1})$. As for
condition $(\ref{condition33})$ of Theorem $(\ref{theoremConvex})$, when
$a_{n}\rightarrow\infty$, it holds
\begin{align*}
nG(a_{n})  &  =ng\left(  a_{n}+\frac{1}{g(a_{n})}\right)  -ng(a_{n})\\
&  =ng(a_{n})+n\frac{g^{\prime}(a_{n})}{g(a_{n})}+o\left(  \frac{g^{\prime
}(a_{n})}{g(a_{n})}\right)  -ng(a_{n})\\
&  =n\frac{g^{\prime}(a_{n})}{g(a_{n})}+o\left(  \frac{g^{\prime}(a_{n}%
)}{g(a_{n})}\right)  =o(n).
\end{align*}
Hence under condition $(\ref{X order's condition 1})$, it holds $nG(a_{n}%
)=o(H(a_{n},\epsilon_{n}))$, which means that condition $(\ref{condition33})$
of Theorem $\ref{theoremConvex}$ holds under condition
$(\ref{X order's condition 1})$, which completes the proof.

\section{Very Large Deviation for Exponential Density Functions associated to
non-convex Functions}

In this section, we pay attention to exponential density functions whose
exponents are non-convex functions. Namely, i.i.d random variables
$X_{1},...,X_{n}$ have common density with
\[
f(x)=c\exp\Big(-\left(  g(x)+q(x)\right)  \Big)
\]
assuming that the convex function $g$ is twice differentiable and $q(x)$ is of
smaller order than $\log g(x)$ for large $x$.

\bigskip

\begin{theorem}
\label{THM2}$X_{1},...,X_{n}$ are i.i.d. real valued random variables with
common density $f(x)=c\exp\left(  -(g(x)+q(x))\right)  $, where $g(x)$ is some
positive convex function on $\mathbb{R}^{+}$ and $g$ is twice differentiable.
Assume \ that $on[X,\infty)$, $g(x)$ is increasing on $[X,\infty)$ and
satisfies
\[
\lim_{x\rightarrow\infty}g(x)/x=\infty.
\]
Let $M(x)$ be some nonnegative continuous function on $\mathbb{R}^{+}$ for
which
\[
-M(x)\leq q(x)\leq M(x)\text{ \ \ \ for all positive }x
\]
together with
\begin{align}
\label{thcond011}M(x)=O\left(  \log g(x)\right)
\end{align}
as $x\rightarrow\infty.$

Let $a_{n}$ be some positive sequence such that $a_{n}\rightarrow\infty$ and
$\epsilon_{n}=o(a_{n})$ be a positive sequence. Assume
\begin{align}
\lim\inf_{n\rightarrow\infty}  &  \frac{\log g(a_{n})}{\log n}%
>0\label{thcond12}\\
\lim_{n\rightarrow\infty}  &  \frac{n\log g\left(  a_{n}+\epsilon_{n}\right)
}{H(a_{n},\epsilon_{n})}=0,\label{thcond13}\\
\lim_{n\rightarrow\infty}  &  \frac{nG(a_{n})}{H(a_{n},\epsilon_{n})}=0,
\label{thcond14}%
\end{align}
where
\begin{align*}
&  H(a_{n},\epsilon_{n})=\min\left(  F_{g_{1}}(a_{n},\epsilon_{n}),F_{g_{2}%
}(a_{n},\epsilon_{n})\right)  -ng(a_{n}),\\
&  G(a_{n})=g\left(  a_{n}+\frac{1}{g(a_{n})}\right)  -g(a_{n}),
\end{align*}
where $F_{g_{1}}(a_{n},\epsilon_{n})$ and $F_{g_{2}}(a_{n},\epsilon_{n})$ are
defined as in Lemma $\ref{Lemma01}$.

Then it holds
\[
P(I|C)\rightarrow1\text{ \ \ when }n\rightarrow\infty.
\]

\end{theorem}

We now provide examples of densities which define r.v's $X_{i}^{\prime}$'s for
which the above Theorem \ref{THM2} applies. These densities appear in a number
of questions pertaining to uniformity in large deviation approximations; see
\cite{JE1995} Ch 6.

\begin{example}
\textbf{Almost Log-concave densities}: $p$ can be written as
\[
p(x)=c(x)\exp-h(x),\text{ \ \ \ \ }x<\infty
\]
with $h$ a convex function, and where for some $x_{0}>0$ and constants
$0<c_{1}<c_{2}<\infty,$ we have%
\[
c_{1}<c(x)<c_{2}\text{ for }x_{0}<x<\infty.
\]
Densities which satisfy the above condition include the Normal, the Gamma, the
hyperbolic density, etc.
\end{example}

\begin{example}
\textbf{Gamma-like densities} are defined through densities of the form
\[
p(x)=c(x)\exp-h(x)
\]
for all $x>0$, with $0<c_{1}<c(x)<c_{2}\leq\infty$ when $x$ is larger than
some $x_{0}>0$ and $h(x)$ is a convex function which satisfies $h(x)=\tau
+h_{1}(x)$ with, for $x_{1}<x_{2}$,
\[
a_{1}\log\frac{x_{2}}{x_{1}}-b_{1}<h_{1}(x_{2})-h_{1}(x_{1})<a_{2}\log
\frac{x_{2}}{x_{1}}-b_{2}%
\]
where $a_{1},a_{2},b_{1}$ and $b_{2}$ are positive constants with $a_{2}<1.$
\end{example}

A wide class of densities for which our results apply is when there exist
constants $x_{0}>0$, $\alpha>0,$ $\tau>0$ and $A$ such that
\[
p(x)=Ax^{\alpha-1}l(x)\exp\left(  -\tau x\right)  \text{ \ \ \ }x>x_{0}
\]
where $l(x)$ is slowly varying at infinity.

\begin{example}
\label{2ex04} \textbf{Almost Log-concave densities 1}: $p$ can be written as
\[
p(x)=c(x)\exp-g(x),\text{ \ \ \ \ }0<x<\infty
\]
with $g$ a convex function, and where for some $x_{0}>0$ and constants
$0<c_{1}<c_{2}<\infty,$ we have%
\[
c_{1}<c(x)<c_{2}\text{ for }x_{0}<x<\infty,
\]
and $g(x)$ is increasing on some interval $[X,\infty)$ and satisfies
\[
\lim_{x\rightarrow\infty}g(x)/x=\infty.
\]
Examples of densities which satisfy the above conditions include the Normal,
the hyperbolic density, etc.
\end{example}

\begin{example}
\label{2ex05} \textbf{Almost Log-concave densities 2}: A wide class of
densities for which our results apply is when there exist constants $x_{0}>0$,
$\alpha>0,$ and $A$ such that
\[
p(x)=Ax^{\alpha-1}l(x)\exp\left(  -g(x)\right)  \text{ \ \ \ }x>x_{0}%
\]
where $l(x)$ is slowly varying at infinity, $g$ a convex function, increasing
on some interval $[X,\infty)$ and satisfies
\[
\lim_{x\rightarrow\infty}g(x)/x=\infty.
\]

\end{example}

\begin{remark}
All density functions in Examples $(\ref{2ex04})$ $(\ref{2ex05})$ satisfy the
assumptions of the above Theorem \ref{THM2} . Also the conditions in Theorem
\ref{THM2} about $a_{n}$ and $\epsilon_{n}$ are the same as those in the
convex case, so that if $g(x)$ is some power function with index larger than
$2$, $\epsilon_{n}$ can go to $0$ more rapidly than $O(1/\log a_{n})$(see
Example $\ref{1ex01}$); If $g(x)$ is of exponential function form,
$\epsilon_{n}$ goes to $0$ more rapidly than any power $1/a_{n}$ (see Example
$\ref{1ex02}$ ).*
\end{remark}

\section{Application}

An extended LDP holds for the partial sum $S_{1}^{n}$ where the i.i.d.
summands $X_{i}$'s are unbounded above whenever
\[
\lim_{n\rightarrow\infty}-\frac{\log P\left(  S_{1}^{n}/n>x_{n}\right)
}{I(x_{n})}=1
\]
holds where $\lim_{n\rightarrow\infty}x_{n}=+\infty.$ In the above display the
Cramer function $I(x)$ is defined for all $x>EX$ through
\[
I(x):=\sup_{t}tx-\log E\exp tX.
\]
Thne following result holds (see \cite{BRMA1994}, Proposition 1.1). Assume
that $X$ is unbounded above and satisfies the Cramer condition. Assume further
that
\begin{equation}
-\log P\left(  X>x)=I(x)(1+o(1))\right)  \label{condBrMa}%
\end{equation}
as $x\rightarrow\infty.$ Then for any sequence $a_{n}$ going to infinity with
$n$ it holds%
\begin{equation}
-\log P\left(  S_{1}^{n}/n>a_{n})=nI(a_{n})(1+o(1))\right)  \label{BrMa}%
\end{equation}
as $n\rightarrow\infty.$ \ It is readily seen that (\ref{condBrMa}) holds in
any of the cases considered in the present paper (see \cite{BRMA1994}, Remark
1.1). See also \cite{BO1978} for a sharp result.

We now consider the local behaviour of a random walk with independent summands
$X_{i}$, $1\leq i\leq n$ which are identically distributed as $X.$ Let $a>EX$
. We consider random paths $T_{n}:=\left(  S_{1}^{1},S_{1}^{2},..,S_{1}%
^{n}\right)  $ which satisfy $\left(  S_{1}^{n}>na\right)  $ hence under
a\textit{ large deviation condition} pertaining to the end value. In the
following result we state that the trajectory $T_{n}$ exhibits a peculiar feature.

Let $k=k_{n}$ be an integer sequence such that $\lim_{n\rightarrow\infty
}k=\infty$ together with $\lim_{n\rightarrow\infty}k/n=0$, and $\alpha
_{n}\rightarrow\infty$ such that
\[
\lim_{n\rightarrow\infty}\frac{na-k\alpha_{k}}{n-k}=\infty.
\]
Denote $A_{k}$ the event%
\[
A_{k}:=\left(  \text{there exists }j\text{, }1\leq j\leq n-k\text{ such that
}\Delta_{j,k}>\alpha_{k}\right)  .
\]

It holds

\begin{proposition}
When $X$ satisfies the hypotheses in Theorem \ref{THM2} it holds
\[
P\left(  \left.  A_{k}\right\vert S_{1}^{n}>na\right)  \rightarrow1.
\]

\end{proposition}

\begin{proof}
The proof is simple and we briefly sketch the argument. Clearly
\begin{align*}
P\left(  \left.  A_{k}\right\vert S_{1}^{n}>na\right)   &  =1-P\left(  \left.
\cap_{j=0}^{\left[  n/k\right]  }\left(  S_{kj+1}^{kj}<k\alpha_{k}\right)
\right\vert S_{1}^{n}>na\right) \\
&  \leq1-\left(  P\left(  \left.  S_{1}^{k}<k\alpha_{k}\right\vert S_{1}%
^{n}>na\right)  \right)  ^{\left[  n/k\right]  }=:1-P.
\end{align*}
Now applying Bayes Theorem and the independence of the r.v's $X_{i}$'s, it
holds%
\[
P\leq\frac{P\left(  S_{k+1}^{n}>\frac{na-k\alpha_{k}}{n-k}\right)  }{P\left(
S_{1}^{n}>na\right)  }.
\]
Under the present hypotheses (\ref{condBrMa}) holds.Using (\ref{BrMa}) in the
numerator and the classical first order LDP result
\[
\log P\left(  S_{1}^{n}>na\right)  =-nI(a)\left(  1+o(1)\right)
\]
in the denominator, it follows that $P\rightarrow0$ as $n\rightarrow\infty$,
which concludes the proof.
\end{proof}

The consequence of Theorem \ref{THM2} is that on this segment of length $k$
where the slope exceeds $\alpha_{k}$ all the summands are of order $\alpha
_{k}$ so that the behaviour of the trajectory is nearly linear. Numerical
evidence confirm the theoretical ones; for very large $a$ and fixed (large)
$n$ , not surprisingly, the \textit{democratic localisation} holds on the
entire trajectory , in accordance with the results in this paper;
therefore\textit{ }$T_{n}$ is nearly a straight line from the origin up to the
point $(n,na_{n}).$ For smaller values of $a$ (typically for $a$ defined
through $P\left(  S_{1}^{n}>na\right)  $ of order $10^{-3}$ the phenomenon
quoted in the above proposition holds: $T_{n}$ consists in a number of oblic
segments. When $n$ is allowed to increase, the segments are longer and longer,
with increasing slope.

\section{Appendix}

\subsection{Proof of Lemma \ref{Lemma01}}

Write $\mathbf{x}=\left(  x_{1},...,x_{n}\right)  \in\mathbb{R}_{+}^{n}$, we
firstly define the following sets. Let for all $k$ between $O$ and $n$
\[
A_{k}:=\left\{  \text{there exist }i_{1},..,i_{k}\text{ such that }x_{i_{j}%
}\geq a_{n}+\epsilon_{n}\text{ for all }j\text{ with }1\leq j\leq k\right\}
\]

and
\[
B_{k}:=\left\{  \text{there exist }i_{1},..,i_{k}\text{ such that }x_{i_{j}%
}\leq a_{n}-\epsilon_{n}\text{ for all }j\text{ with }1\leq j\leq k\right\}
.
\]
Define%
\[
A=\bigcup\lim_{k=1}^{n}A_{k}%
\]
and
\[
B=\bigcup\lim_{k=1}^{n}B_{k}.
\]
It then holds
\[
I^{c}=A\cup B.
\]
It follows that
\begin{align*}
I_{g}(I^{c}\cap C)  &  =I_{g}\left(  (A\cup B)\cap C\right)  =\inf
_{\mathbf{x}\in(A\cap C)\cup(B\cap C)}I_{g}(\mathbf{x})\\
&  =\min\left(  I_{g}(A\cap C),I_{g}(B\cap C)\right)  .
\end{align*}
Thus we may calculate the minimum values of $\ $both $I_{g}(A\cap C)$ and
$I_{g}(B\cap C)$ respectively, and finally $I_{g}(I^{c}\cap C)$.

\textbf{Step 1:} In this step we prove that
\begin{equation}
I_{g}(A\cap C)=F_{g_{1}}(a_{n},\epsilon_{n}). \label{I(AetC)}%
\end{equation}
Let $\mathbf{x}:=\left(  x_{1},...,x_{n}\right)  $ belong to $A\cap C$ and
assume that $I_{g}(A\cap C)=I_{g}(\mathbf{x})$. Without loss of generality,
assume that the $x_{i}$'s are ordered ascendently, $x_{1}\leq...\leq x_{i}\leq
x_{i+1\leq},...\leq x_{n}$ and let $i$ and $k:=n-i$ with $1\leq i\leq n$ such
that
\[
\overbrace{x_{1}\leq...\leq x_{i}}^{n-k}<a_{n}+\epsilon_{n}\leq\overbrace
{x_{i+1}\leq...\leq x_{n}}^{k}.
\]
We first claim that $k<n$. Indeed let $\mathbf{y:=}\left(  y_{1}%
=a_{n}-\epsilon_{n},y_{2}=...=y_{n-1}=a_{n}+\epsilon_{n}\right)  $ which
clearly belongs to $A\cap C.$ For this \textbf{\ }$\mathbf{y}$ it holds
$I_{g}(\mathbf{y})=(n-1)g(a_{n}+\epsilon_{n})+g(a_{n}-\epsilon_{n})$ which is
strictly smaller than $ng(a_{n}+\epsilon_{n})=I_{g}(A_{n}\cap C)$ for large
$n.$ We have proved that $\mathbf{x}$ does not belong to $A_{n}\cap C.$

Let $\alpha_{i+1},...,\alpha_{n}$ be nonnegative, and write $x_{i+1}%
,...,x_{n}$ as
\[
x_{i+1}=a_{n}+\epsilon_{n}+\alpha_{i+1},...,x_{n}=a_{n}+\epsilon_{n}%
+\alpha_{n}.
\]
Under condition $(C)$, it holds
\begin{align*}
x_{1}+...+x_{i}  &  \geq na_{n}-\left(  x_{i+1}+...+x_{n}\right) \\
&  =na_{n}-k(a_{n}+\epsilon_{n})-\left(  \alpha_{i+1}+...+\alpha_{n}\right)  .
\end{align*}
Applying Jensen's inequality to the convex function $g$, we have
\begin{align*}
\sum_{i=1}^{n}g(x_{i})  &  =\left(  g(x_{i+1})+...+g(x_{n})\right)  +\left(
g(x_{1})+...+g(x_{i})\right) \\
&  \geq\left(  g(x_{i+1})+...+g(x_{n})\right)  +(n-k)g(x^{\ast}),
\end{align*}
where equality holds when $x_{1}=...=x_{i}=x^{\ast}$, with
\[
x^{\ast}=\frac{na_{n}-k(a_{n}+\epsilon_{n})-\left(  \alpha_{i+1}%
+...+\alpha_{n}\right)  }{n-k}.
\]
Define now the function function $(\alpha_{i+1},...,\alpha_{n},k)\rightarrow$
$f(\alpha_{i+1},...,\alpha_{n},k)$ through
\begin{align*}
f(\alpha_{i+1},...,\alpha_{n},k)  &  =g(x_{i+1})+...+g(x_{n})+(n-k)g(x^{\ast
})\\
&  =g(a_{n}+\epsilon_{n}+\alpha_{i+1})+...+g(a_{n}+\epsilon_{n}+\alpha
_{n})+(n-k)g(x^{\ast}).
\end{align*}
Then $I_{g}(A\cap C)$ is given by
\[
I_{g}(A\cap C)=\inf_{\alpha_{i+1},...,\alpha_{n}\geq0,1\leq k\leq n}%
f(\alpha_{i+1},...,\alpha_{n},k).
\]
We now obtain (\ref{I(AetC)}) through the properties of the function $f.$
Using $(\ref{lemma1formula21})$, the first order partial derivative of
$f(\alpha_{i+1},...,\alpha_{n},k)$ with respect to $\alpha_{i+1}$ is
\[
\frac{\partial f(\alpha_{i+1},...,\alpha_{n},k)}{\partial\alpha_{i+1}%
}=g^{\prime}(a_{n}+\epsilon_{n}+\alpha_{i+1})-g^{\prime}(x^{\ast})>0,
\]
where the inequality holds since $g(x)$ is strictly convex and $a_{n}%
+\epsilon_{n}+\alpha_{i+1}>x^{\ast}$. Hence $f(\alpha_{i+1},...,\alpha_{n},k)$
is an increasing function with respect to $\alpha_{i+1}$. This implies that
the minimum value of $f$ is attained when $\alpha_{i+1}=0$. In the same way,
we have $\alpha_{i+1}=...=\alpha_{n}=0.$ Therefore it holds
\[
I_{g}(A\cap C)=\inf_{1\leq k\leq n}f(\mathbf{0},k),
\]
with
\[
f(\mathbf{0},k)=kg(a_{n}+\epsilon_{n})+(n-k)g(x_{0}^{\ast}),
\]
where
\[
x_{0}^{\ast}=a_{n}-\frac{k}{n-k}\epsilon_{n}.
\]
The function $y\rightarrow f(\mathbf{0},y)$ with $0<y<n$ is increasing with
respect to $y$, since
\begin{align*}
\frac{\partial f(\mathbf{0},y)}{\partial y}  &  =g(a_{n}+\epsilon_{n}%
)-g(x_{0}^{\ast})-\frac{n\epsilon_{n}}{n-y}g^{\prime}(x_{0}^{\ast})\\
&  =\frac{n\epsilon_{n}}{n-y}\left(  \frac{g(a_{n}+\epsilon_{n})-g(x_{0}%
^{\ast})}{a_{n}+\epsilon_{n}-x_{0}^{\ast}}-g^{\prime}(x_{0}^{\ast})\right)
>0,
\end{align*}
due to the convexity of $g(x)$ and $a_{n}+\epsilon_{n}>x_{0}^{\ast}$. Hence
$f(\mathbf{0},k)$ is increasing with respect to $k;$ the minimal value of
$f(\mathbf{0},k)$attains with $k=1$. Thus we have
\[
I_{g}(A\cap C)=f(\mathbf{0},1)=F_{g_{1}}(a_{n},\epsilon_{n})
\]
which proves $(\ref{I(AetC)})$.

\bigskip

\textbf{Step 2:} In this step, we follow the same proof as above and prove
that
\[
I_{g}(B\cap C)=F_{g_{2}}(a_{n},\epsilon_{n}).
\]
With $\mathbf{x}$ defined through $I_{g}(\mathbf{x}):=I_{g}(B\cap C)$ with the
coordinates of $\mathbf{x}$ ranked in ascending order, with $j$ such that
$1\leq j\leq n$ and
\[
\overbrace{x_{1}\leq...\leq x_{j}}^{j}<a_{n}+\epsilon_{n}\leq\overbrace
{x_{j+1}\leq...\leq x_{n}}^{n-j}%
\]
we obtain $j<n$ through the same argument as above. Denote $x_{1},...,x_{j}$
by
\[
x_{1}=a_{n}-\epsilon_{n}-\alpha_{1},...,x_{n}=a_{n}-\epsilon_{n}-\alpha_{j},
\]
where $\alpha_{1},...,\alpha_{j}$ are nonnegative. Under condition $(C)$, it
holds
\begin{align*}
x_{j+1}+...+x_{n}  &  \geq na_{n}-\left(  x_{1}+...+x_{j}\right) \\
&  =na_{n}-j(a_{n}-\epsilon_{n})+\left(  \alpha_{1}+...+\alpha_{j}\right)  .
\end{align*}
Using Jensen's inequality to\ the convex function $g(x)$, we have
\begin{align*}
\sum_{i=1}^{n}g(x_{i})  &  =\left(  g(x_{1})+...+g(x_{j})\right)  +\left(
g(x_{j+1})+...+g(x_{n})\right) \\
&  \geq\left(  g(x_{1})+...+g(x_{j})\right)  +(n-j)g(x^{\sharp}),
\end{align*}
where the equality holds when $x_{j+1}=...=x_{n}=x^{\sharp}$, with
\[
x^{\sharp}=\frac{na_{n}-j(a_{n}-\epsilon_{n})+\left(  \alpha_{1}%
+...+\alpha_{j}\right)  }{n-j}.
\]
Define the function $(\alpha_{i+1},...,\alpha_{n},k)\rightarrow f(\alpha
_{i+1},...,\alpha_{n},k)$ through
\begin{align*}
f(\alpha_{1},...,\alpha_{j},j)  &  =g(x_{1})+...+g(x_{j})+(n-j)g(x^{\sharp})\\
&  =g(a_{n}-\epsilon_{n}-\alpha_{1})+...+g(a_{n}-\epsilon_{n}-\alpha
_{j})+(n-j)g(x^{\sharp}),
\end{align*}
then $I_{g}(A\cap C)$ is given by
\[
I_{g}(A\cap C)=\inf_{\alpha_{1},...,\alpha_{j}\geq0,1\leq j\leq n}f(\alpha
_{1},...,\alpha_{j},j).
\]

Using $(\ref{lemma1formula22})$, the first order partial derivative of
$f(\alpha_{1},...,\alpha_{j},j)$ with respect to $\alpha_{1}$ is
\[
\frac{\partial f(\alpha_{1},...,\alpha_{j},j)}{\partial\alpha_{1}}=-g^{\prime
}(a_{n}-\epsilon_{n}-\alpha_{1})+g^{\prime}(x^{\sharp})>0,
\]
where the inequality holds since $g(x)$ is convex and $a_{n}-\epsilon
_{n}-\alpha_{1}<x^{\sharp}$. Hence $f(\alpha_{1},...,\alpha_{j},j)$ is
increasing with respect to $\alpha_{1}$. This yields
\[
\alpha_{1}=...=\alpha_{j}=0.
\]
Therefore it holds
\[
I_{g}(B\cap C)=\inf_{1\leq k\leq n}f(\mathbf{0},j),
\]
with
\[
f(\mathbf{0},j)=jg(a_{n}-\epsilon_{n})+(n-j)g(x_{0}^{\sharp}),
\]
where
\[
x_{0}^{\sharp}=a_{n}+\frac{j}{n-j}\epsilon_{n}.
\]
The function $y\rightarrow f(\mathbf{0},y)$ with $0<y<n$ is increasing with
respect to $y$, since
\begin{align*}
\frac{\partial f(\mathbf{0},y)}{\partial y}  &  =g(a_{n}-\epsilon_{n}%
)-g(x_{0}^{\sharp})+\frac{n\epsilon_{n}}{n-j}g^{\prime}(x_{0}^{\sharp})\\
&  =\frac{n\epsilon_{n}}{n-y}\left(  g^{\prime}(x_{0}^{\sharp})-\frac
{g(x_{0}^{\sharp})-g(a_{n}-\epsilon_{n})}{x_{0}^{\sharp}-(a_{n}-\epsilon_{n}%
)}\right)  >0,
\end{align*}
by is convexity of $g$ ; in the above display $x_{0}^{\sharp}>a_{n}%
-\epsilon_{n}$. Hence $f(\mathbf{0},k)$ is increasing with respect to $k.$
Thus we have
\[
I_{g}(B\cap C)=f(\mathbf{0},1)=F_{g_{{}}}(a_{n},\epsilon_{n})
\]
which proves the claim.

Thus the proof is completed using $(\ref{I(AetC)})$ and $(\ref{I(BetC})$.

\subsection{Proof of Theorem \ref{theoremConvex}}

For $x=(x_{1},...,x_{n})\in\mathbb{R}_{+}^{n}$, define
\[
S_{g}(r)=\left\{  x:\sum_{1\leq i\leq n}g(x_{i})\leq r\right\}  .
\]
Then for any Borel set $A$ in $\mathbb{R}^{n}$ it holds
\begin{align*}
P(A)=  &  \int_{A}\exp\left(  -\sum_{1\leq i\leq n}p(x_{i})\right)
dx_{1},...,dx_{n}\\
&  =\exp(-I_{g}(A))\int_{A}dx_{1},...,dx_{n}\int{\Large 1}_{\left[
\sum_{1\leq i\leq n}g(x_{i})-I_{g}(A),\infty\right)  }(s)e^{-s}ds\\
&  =\exp(-I_{g}(A))\int_{0}^{\infty}Volume(A\cap S_{g}(I_{g}(A)+s))e^{-s}ds.
\end{align*}
The proof is divided in three steps.

\textbf{Step 1:} We prove that
\begin{align}
\label{convexlowerbound}P(C)\geq c^{n}\exp\left(  -I_{g}(C)-\tau_{n}-n\log
g(a_{n})\right)  .
\end{align}
where
\begin{align}
\label{theConv10}\tau_{n}=ng\left(  a_{n}+\frac{1}{g(a_{n})}\right)
-ng(a_{n}).
\end{align}

By convexity of $\ $\ the function $g$, and using \ condition $(C)$, applying
Jensen's inequality, with $x_{1}=...=x_{n}=a_{n}$ it holds
\[
I_{g}(C)=ng(a_{n}).
\]
We now consider the largest lower bound for
\[
\log Volume\left(  C\cap S_{g}(I_{g}(C)+\tau_{n})\right)  .
\]
\newline Denote $B=\left\{  \mathbf{x}:x_{i}\in\lbrack a_{n},a_{n}+\frac
{1}{g(a_{n})}]\right\}  $, $S_{g}(I_{g}(C)+\tau_{n})=\{\mathbf{x}:\sum
_{i=1}^{n}g(x_{i})\leq ng(a_{n})+\tau_{n}\}$.

For large $n$ and any $\mathbf{x:=}\left(  x_{1},..,x_{n}\right)  $ in $B$, it
holds
\[
\sum_{i=1}^{n}g(x_{i})\leq\sum_{i=1}^{n}g\Big(a_{n}+\frac{1}{g(a_{n}%
)}\Big)=ng\Big(a_{n}+\frac{1}{g(a_{n})}\Big)=ng(a_{n})+\tau_{n},
\]
where we used the fact that $g$ is an increasing function for large argument.
Hence
\[
B\subset S_{g}(I_{g}(C)+\tau_{n}).
\]
It follows that
\begin{align}
\label{lower bound p}\log Volume\left(  C\cap S_{g}(I_{g}(C)+\tau_{n})\right)
\geq\log Volume(B)=\log\left(  \frac{1}{g(a_{n})}\right)  ^{n}=-n\log g(a_{n})
\end{align}
which in turn using $(\ref{Integration p})$ and $(\ref{lower bound p}%
)$,implies
\begin{align*}
\log P(C):=  &  \log\int_{C}\exp\left(  -\sum_{1\leq i\leq n}g(x_{i})\right)
dx_{1},...,dx_{n}\\
&  \geq\log\left(  \exp(-I_{g}(C))\int_{\tau_{n}}^{\infty}Volume(C\cap
S_{g}(I_{g}(C)+s))e^{-s}ds\right) \\
&  \geq-I_{g}(C)-\tau_{n}+\log Volume(C\cap S_{g}(I_{g}(C)+\tau_{n}))\\
&  \geq-I_{g}(C)-\tau_{n}-n\log g(a_{n}),
\end{align*}
This proves the claim.

\textbf{Step 2:} In this step, we prove that%
\begin{align}
\label{convexupperbound}P(I^{c}\cap C)\leq c^{n}\exp\left(  -I_{g}(I^{c}\cap
C)+n\log I_{g}(I^{c}\cap C)+\log(n+1)\right)  .
\end{align}

For any Borel set $A$ in $\mathbb{R}^{n}$ it holds , for positive $s$,
$\ $\ let
\[
S_{g}(I_{g}(A)+s)=\left\{  \mathbf{x}:\sum_{1\leq i\leq n}g(x_{i})\leq
I_{g}(A)+s\right\}
\]
and
\[
F=\left\{  \mathbf{x}:g(x_{i})\leq I_{g}(A)+s,i=1,...,n\right\}  .
\]
It holds.%
\[
S_{g}(I_{g}(A)+s)\subset F.
\]

Since $\lim_{x\rightarrow\infty}g(x)/x=+\infty$%

\[
F\subset\{\mathbf{x}:x_{i}\leq(I_{g}(A)+s),i=1,...,n\},
\]
which yields
\[
S_{g}(I_{g}(A)+s)\subset\{\mathbf{x}:x_{i}\leq(I_{g}(A)+s),i=1,...,n\},
\]
from which we obtain
\[
Volume(A\cap S_{g}(I_{g}(A)+s))\leq Volume(S_{g}(I_{g}(A)+s))\leq
(I_{g}(A)+s)^{n}.
\]
With this inequality, the upper bound of integration $(\ref{Integration p})$
can be given when $a_{n}\rightarrow\infty.$%
\begin{align*}
\log P(A)=  &  \log\int_{A}\exp\left(  -\sum_{1\leq i\leq n}g(x_{i})\right)
dx_{1},...,dx_{n}\\
&  =-I_{g}(A)+\log\int_{0}^{\infty}Volume(A\cap S_{g}(I_{g}(A)+s))e^{-s}ds\\
&  \leq-I_{g}(A)+\log\int_{0}^{\infty}\left(  I_{g}(A)+s\right)  ^{n}e^{-s}ds,
\end{align*}
with integrating repeatedly by parts it holds
\begin{align}
&  \int_{0}^{\infty}\left(  I_{g}(A)+s\right)  ^{n}e^{-s}ds\label{th191}\\
&  =I_{g}(A)^{n}+n\int_{0}^{\infty}\left(  I_{g}(A)+s\right)  ^{n-1}%
e^{-s}ds\nonumber\\
&  =I_{g}(A)^{n}+nI_{g}(A)^{n-1}+n(n-1)\int_{0}^{\infty}\left(  I_{g}%
(A)+s\right)  ^{n-2}e^{-s}ds\nonumber\\
&  \leq(n+1)I_{g}(A)^{n},\nonumber
\end{align}
hence we have
\begin{align*}
&  \log\int_{A}\exp\left(  -\sum_{1\leq i\leq n}g(x_{i})\right)
dx_{1},...,dx_{n}\\
&  \leq-I_{g}(A)+\log\left(  (n+1)I_{g}(A)^{n}\right) \\
&  =-I_{g}(A)+n\log I_{g}(A)+\log(n+1).
\end{align*}
Replace $A$ by $I^{c}\cap C$. We then obtain%
\[
P(I^{c}\cap C)\leq c^{n}\exp\left(  -I_{g}(I^{c}\cap C)+n\log I_{g}(I^{c}\cap
C)+\log(n+1)\right)
\]
as sought.

\textbf{Step 3:} In this step, we will complete the proof , showing that
\[
\lim_{a_{n}\rightarrow\infty}\frac{P(I^{c}\cap C)}{P(C)}=0.
\]
By Lemma $\ref{Lemma01}$,
\[
I_{g}(I^{c}\cap C)=\min\left(  F_{g_{1}}(a_{n},\epsilon_{n}),F_{g_{2}}%
(a_{n},\epsilon_{n})\right)  .
\]
Using $(\ref{convexlowerbound})$ and $(\ref{convexupperbound})$ it holds
\[
\frac{P(I^{c}\cap C)}{P(C)}\leq\exp\left(  -H(a_{n},\epsilon_{n})+n\log
I_{g}(I^{c}\cap C)+\tau_{n}+n\log g(a_{n})+\log(n+1)\right)  .
\]
Under conditions $(\ref{condition33})$, by $(\ref{theConv10})$ when
$a_{n}\rightarrow\infty$, we have
\[
\frac{\tau_{n}}{H(a_{n},\epsilon_{n})}=\frac{nG(a_{n})}{H(a_{n},\epsilon_{n}%
)}\longrightarrow0,
\]
Using conditions $(\ref{condition31})$ and $(\ref{condition32})$, when
$a_{n}\rightarrow\infty$,
\[
\frac{n\log g(a_{n})}{H(a_{n},\epsilon_{n})}\longrightarrow0,\qquad
\emph{and}\qquad\frac{\log(n+1)}{H(a_{n},\epsilon_{n})}\longrightarrow0.
\]
As to the term $n\log I_{g}(I^{c}\cap C)$, we have
\begin{align*}
n\log I_{g}(I^{c}\cap C)  &  =n\min\left(  F_{g_{1}}(a_{n},\epsilon
_{n}),F_{g_{2}}(a_{n},\epsilon_{n})\right) \\
&  \leq n\log\left(  ng(a_{n}+\epsilon_{n})\right) \\
&  =n\log n+n\log g\left(  a_{n}+\epsilon_{n}\right)  .
\end{align*}
Under condition $(\ref{condition32})$, when $a_{n}\rightarrow\infty$, $n\log
g\left(  a_{n}+\epsilon_{n}\right)  $ is of small order with respect to
$H(a_{n},\epsilon_{n})$ as $n$ tends to infinity. Under condition
$(\ref{condition31})$, for $a_{n}$ large enough, there exists some positive
constant $Q$ such that $\log n\leq Q\log g(a_{n}).$ Hence we have%
\[
n\log n\leq Qn\log g(a_{n})
\]
which under condition $(\ref{condition32})$, yields that $n\log n$ is
negligible with respect to $H(a_{n},\epsilon_{n})$. Hence when $a_{n}%
\rightarrow\infty$, it holds
\[
\frac{n\log\left(  I_{g}(I^{c}\cap C)\right)  }{H(a_{n},\epsilon_{n}%
)}\longrightarrow0.
\]
Further, $(\ref{theconv11})$, $(\ref{theconv12})$ and $(\ref{theconv13})$ make
$(\ref{mainresult})$ hold. This completes the proof.

\subsection{Proof of Theorem \ref{THM2}}

The proof is is the same vein as that of Theorem \ref{theoremConvex}; some
care has to be taken in order to get similar bounds as developped in the
convex case.

Denote $\mathbf{x}=(x_{1},...,x_{n})$ in $\mathbb{R}^{n}$ and, for a Borel set
$A\in\mathbb{R_{+}}^{n}$ define
\[
I_{g,q}(A)=\inf_{\mathbf{x}\in A}I_{g,q}(\mathbf{x}),
\]
where
\[
I_{g,q}(\mathbf{x}):=\sum_{1\leq i\leq n}\left(  g(x_{i})+q(x_{i})\right)  .
\]
Also for any positive $r$ define
\[
S_{g,q}(r)=\left\{  \mathbf{x}:\sum_{1\leq i\leq n}\left(  g(x_{i}%
)+q(x_{i})\right)  \leq r\right\}  .
\]
Then it holds
\begin{align}
\label{th11} &  P(A)=\int_{A}\exp\left(  -\sum_{1\leq i\leq n}\left(
g(x_{i})+q(x_{i})\right)  \right)  dx_{1},...,dx_{n}\nonumber\\
&  =\exp(-I_{{g,q}}(A))\int_{A}dx_{1},...,dx_{n}\int{\Large 1}_{\left[
{\sum_{1\leq i\leq n}(g(x_{i})+q(x_{i}))}-I_{g,q}(A),\infty\right)  }%
(s)e^{-s}ds\nonumber\\
&  =\exp(-I_{g,q}(A))\int_{0}^{\infty}Volume(A\cap S_{g,q}(I_{g,q}%
(A)+s))e^{-s}ds.
\end{align}
\textbf{Step 1:} In this step we prove that
\[
I_{g,q}(C)\geq I_{g_{1}}(C)\geq nh(a_{n})=ng(a_{n})-nN\log g(a_{n}).
\]

For large $x$ it holds
\begin{equation}
g(x)-M(x)\leq g(x)+q(x)\leq g(x)+M(x). \label{th10}%
\end{equation}
Set $g_{1}(x)=g(x)-M(x)$ and $g_{2}(x)=g(x)+M(x)$, then it follows
\begin{equation}
I_{g_{1}}(C)\leq I_{g,q}(C)\leq I_{g_{2}}(C). \label{th100}%
\end{equation}
In the same way, it holds
\begin{equation}
I_{g_{1}}(I^{c}\cap C)\leq I_{g,q}(I^{c}\cap C)\leq I_{g_{2}}(I^{c}\cap C).
\label{th101}%
\end{equation}
By condition $(\ref{thcond011})$, there exists some sufficiently large
positive $y_{0}$ and some positive constant $N$ such that for $x\in\lbrack
y_{0},\infty)$
\begin{equation}
M(x)\leq N\log g(x). \label{th1011}%
\end{equation}
Set $r(x)=g(x)-N\log g(x)$, the second order derivative of $r(x)$ is
\[
r^{\prime\prime}(x)=g^{\prime\prime}(x)\left(  1-\frac{N}{g(x)}\right)
+\frac{N\left(  g^{\prime}(x)\right)  ^{2}}{g^{2}(x)},
\]
where the second term is positive. The function $g$ is increasing on some
interval $[X,\infty)$ where we also have $g(x)>x.$ Hence there exists some
$y_{1}\in\lbrack X,\infty)$ such that s $g(x)>N$ when $x\in\lbrack
y_{1},\infty)$. This implies that $r^{\prime\prime}(x)>0$ and $r^{\prime
}(x)>0$ and therefore $r(x)$ is convex and increasing on $[y_{1},\infty)$.

In addition, $M(x)$ is bounded on any finite interval; there exists some
$y_{2}\in\lbrack y_{1},\infty)$ such that for all $x\in(0,y_{2})$
\begin{equation}
M(x)\leq N\log g(y_{2}). \label{th430}%
\end{equation}
The function $g$ is convex and increasing on $[y_{2},\infty)$. Thus there
exists $y_{3}$ such that
\begin{equation}
g^{\prime}(y_{3})>2g^{\prime}(y_{2})\quad and\quad g(y_{3})>2N. \label{th4001}%
\end{equation}
We now construct a function $h$ as follows. Let
\begin{equation}
h(x)=r(x)\mathbf{1}_{[y_{3},\infty)}(x)+s(x)\mathbf{1}_{(0,y_{3})}(x),
\label{th4100}%
\end{equation}
where $s(x)$ is defined by
\begin{equation}
s(x)=r(y_{3})+r^{\prime}(y_{3})(x-y_{3}). \label{th41}%
\end{equation}

We will show that
\begin{equation}
g_{1}(x)\geq h(x) \label{th42}%
\end{equation}
for $x\in(0,\infty)$ .

If $x\in\lbrack y_{3},\infty)$, then by $(\ref{th1011})$, it holds
\begin{equation}
h(x)=r(x)=g(x)-N\log g(x)\leq g(x)-M(x)=g_{1}(x). \label{th410}%
\end{equation}
If $x\in(y_{2},y_{3})$, using $(\ref{th41})$, we have
\begin{equation}
s(x)\leq r(x)=g(x)-N\log g(x)\leq g(x)-M(x)=g_{1}(x), \label{th411}%
\end{equation}
where the first inequality comes from the convexity of $r(x)$. We now show
that $(\ref{th42})$ holds when $x\in(0,y_{2}]$ if $y_{3}$ is large enough. For
this purpose, set
\[
t(x)=g(x)-s(x)-N\log g(y_{2}).
\]
Take the first order derivative of $t$ and use the convexity of $g$ on
$(0,y_{2}].$ We have
\begin{align*}
t^{\prime}(x)  &  =g^{\prime}(x)-s^{\prime}(x)=g^{\prime}(x)-r^{\prime}%
(y_{3})=g^{\prime}(x)-\left(  g^{\prime}(y_{3})-\frac{Ng^{\prime}(y_{3}%
)}{g(y_{3})}\right) \\
&  =g^{\prime}(x)-\left(  1-\frac{N}{g(y_{3})}\right)  g^{\prime}(y_{3})\leq
g^{\prime}(y_{2})-\left(  1-\frac{N}{g(y_{3})}\right)  g^{\prime}(y_{3})\\
&  <\frac{1}{2}g^{\prime}(y_{3})-\left(  1-\frac{N}{g(y_{3})}\right)
g^{\prime}(y_{3})<0,
\end{align*}
where the inequalities in the last line hold from $(\ref{th4001})$. Therefore
$t$ is decreasing on $(0,y_{2}].$ It follows that
\[
t(x)\geq t(y_{2})=g(y_{2})-N\log g(y_{2})-s(y_{2})\geq g(y_{2})-N\log
g(y_{2})-r(y_{2})=0,
\]
which, together with $(\ref{th430})$, yields, when $x\in(0,y_{2}]$
\[
g_{1}(x)=g(x)-M(x)\geq g(x)-N\log g(y_{2})\geq s(x).
\]
Together with $(\ref{th410}),(\ref{th411})$, this last display means that
$(\ref{th42})$ holds.

We now prove that $h$ is a convex function on on $(0,\infty)$.; indeed for $x
$ such that $0<x\leq y_{3}$, $h^{\prime\prime}(x)=0$, and if $x>y_{3}$,
$h^{\prime\prime}(x)=r^{\prime\prime}(x)>0$. The left derivative of $h(x)$ at
$y_{3}$ is $h^{\prime}(y_{3}^{-})=r^{\prime}(y_{3})$, and it is obvious that
the right derivative of $h(x)$ at $y_{3}$ is also $h^{\prime}(y_{3}%
^{+})=r^{\prime}(y_{3});$ hence $h$ is derivable at $y_{3}$ and $h^{\prime
}(y_{3})=r^{\prime}(y_{3})$, hence $h^{\prime\prime}(y_{3})=r^{\prime\prime
}(y_{3})>0$. This shows that $h$ is convex on $(0,\infty)$.

Now under condition $(C)$, using the convexity of $h$ and $(\ref{th42})$, it
holds
\[
I_{g_{1}}(\mathbf{x})=\sum_{i=1}^{n}\left(  g(x_{i})-M(x_{i})\right)  \geq
\sum_{i=1}^{n}h(x_{i})\geq nh\left(  \frac{\sum_{i=1}^{n}x_{i}}{n}\right)
=nh(a_{n}).
\]
Using $(\ref{th100})$, we obtain the lower bound of $I_{g,q}(C)$ under
condition $(C)$ for $a_{n}$ large enough (say, $a_{n}>y_{3}$)
\begin{equation}
I_{g,q}(C)\geq I_{g_{1}}(C)\geq nh(a_{n})=nr(a_{n})=ng(a_{n})-nN\log g(a_{n}).
\label{the21}%
\end{equation}

\textbf{Step 2:} In this step, we will show that the following lower bound of
$P(C)$ holds
\begin{equation}
P(C)\geq c^{n}\exp\left(  -I_{g,q}(C)-\tau_{n}-n\log g(a_{n})\right)  ,
\label{th12}%
\end{equation}
where $\tau_{n}$ is defined by
\begin{align}
\tau_{n}  &  =ng\left(  a_{n}+\frac{1}{g(a_{n})}\right)  -ng(a_{n})+nN\log
g\left(  a_{n}+\frac{1}{g(a_{n})}\right)  +nN\log g(a_{n})\label{th51}\\
&  =nG(a_{n})+nN\log g(a_{n})+nN\log g\left(  a_{n}+\frac{1}{g(a_{n})}\right)
.\nonumber
\end{align}

Denote $B=\left\{  \mathbf{x}:x_{i}\in\lbrack a_{n},a_{n}+\frac{1}{g(a_{n}%
)}].\right\}  .$ If $\mathbf{x}\in B$, by $(\ref{th1011})$, which holds for
large $n$ (say, $a_{n}>y_{3}$ and assuming that $g$ is an increasing function
on $\left(  y_{3},\infty\right)  $), we have
\begin{align*}
I_{g,q}(\mathbf{x})  &  \leq\sum_{i=1}^{n}\left(  g(x_{i})+M(x_{i})\right)
\leq\sum_{i=1}^{n}\left(  g(x_{i})+N\log g(x_{i})\right) \\
&  \leq\sum_{i=1}^{n}\left(  g\left(  a_{n}+\frac{1}{g(a_{n})}\right)  +N\log
g\left(  a_{n}+\frac{1}{g(a_{n})}\right)  \right) \\
&  =ng\left(  a_{n}+\frac{1}{g(a_{n})}\right)  +nN\log g\left(  a_{n}+\frac
{1}{g(a_{n})}\right) \\
&  =\tau_{n}+ng(a_{n})-nN\log g(a_{n})\leq\tau_{n}+I_{g,q}(C),
\end{align*}
where the last inequality holds from $(\ref{the21})$. Since $B\subset C$, we
have
\[
B\subset C\cap S_{g,q}(I_{g,q}(C)+\tau_{n}).
\]
Now we may obtain the lower bound
\begin{equation}
\log Volume\left(  C\cap S_{g,q}(I_{g,q}(C)+\tau_{n})\right)  \geq\log
Volume(B)=-n\log g(a_{n}). \label{theoremnonconvex11}%
\end{equation}
Using $(\ref{th11})$ and $(\ref{theoremnonconvex11})$, it holds
\begin{align*}
&  \log\int_{C}\exp\left(  -\sum_{1\leq i\leq n}\left(  g(x_{i})+q(x_{i}%
)\right)  \right)  dx_{1},...,dx_{n}\\
&  =-I_{g,q}(C)+\log\int_{0}^{\infty}Volume(C\cap S_{g,q}(I_{g,q}%
(C)+s))e^{-s}ds\\
&  \geq-I_{g,q}(C)+\log\int_{\tau_{n}}^{\infty}Volume(C\cap S_{g,q}%
(I_{g,q}(C)+\tau_{n}))e^{-s}ds\\
&  \geq-I_{g,q}(C)-\tau_{n}-n\log g(a_{n}),
\end{align*}
so $(\ref{th12})$ holds.

\bigskip

\textbf{Step 3:} We prove that
\begin{align}
\label{th13}P(I^{c}\cap C)\leq c^{n}\exp\left(  -I_{g,q}(I^{c}\cap C)+n\log
I_{g}(I^{c}\cap C)+\log(n+1)+n\log2\right)  .
\end{align}

For any Borel set $A$ in $\mathbb{R}^{n}$ and any positive $s$,%
\[
S_{g,q}(I_{g,q}(A)+s)=\left\{  \mathbf{x}:\sum_{1\leq i\leq n}\left(
g(x_{i})+q(x_{i})\right)  \leq I_{g,q}(A)+s\right\}
\]
is included in $\left\{  \mathbf{x}:g(x_{i})+q(x_{i})\leq I_{g,q}%
(A)+s,i=1,...,n\right\}  $ which\ in turn is included in $F=\{\mathbf{x}%
:g(x_{i})-M(x_{i})\leq(I_{g,q}(A)+s),i=1,...,n\}$ by $(\ref{th10})$.

Set $H=\{\mathbf{x:=}\left(  x_{1},..,x_{n}\right)  :x_{i}\leq2(I_{g,q}%
(A)+s),i=1,...,n\}$, we will show it holds for $a_{n}$ large enough
\[
F\subset H.
\]
Suppose that for some $\mathbf{x:=}\left(  x_{1},..,x_{n}\right)  $ in $F$
,some $x_{i}$ is larger than $2(I_{g,q}(A)+s)$. For $a_{n}$ large enough, by
$(\ref{the21})$, it holds
\begin{align*}
x_{i}  &  \geq2(I_{g,q}(A)+s)\geq2\left(  ng(a_{n})-nN\log g(a_{n})\right) \\
&  >2\left(  ng(a_{n})-\frac{1}{4}ng(a_{n})\right)  =\frac{3}{2}ng(a_{n}).
\end{align*}
Since $\frac{3}{2}ng(a_{n})\geq\frac{3}{2}na_{n}$ for large $n$, by
$(\ref{th1011})$ and since $x\rightarrow g(x)-N\log g(x)$ is increasing, we
have
\begin{align*}
g(x_{i})-M(x_{i})  &  \geq g(x_{i})-N\log g(x_{i})\geq g\left(  2(I_{g,q}%
(A)+s)\right)  -N\log g\left(  2(I_{g,q}(A)+s)\right) \\
&  >g\left(  2(I_{g,q}(C)+s)\right)  -\frac{1}{2}g\left(  2(I_{g,q}%
(C)+s)\right) \\
&  \geq\frac{1}{2}\left(  2(I_{g,q}(C)+s)\right)  =I_{g,q}(C)+s.
\end{align*}
Therefore since $\mathbf{x}\in F$, $x_{i}\leq2(I_{g,q}(A)+s)$ for every $i$,
which implicates that $(\ref{the22})$ holds. Thus we have
\[
S_{g,q}(I_{g,q}(A)+s)\subset H,
\]
from which we deduce that
\begin{align*}
Volume\left(  A\cap S_{g,q}(I_{g,q}(A)+s)\right)   &  \leq Volume\left(
S_{g,q}(I_{g,q}(A)+s)\right) \\
&  \leq Volume(H)=2^{n}(I_{g,q}(A)+s)^{n}.
\end{align*}
With this inequality, the upper bound of integration $(\ref{th11})$ can be
given when $a_{n}\rightarrow\infty$ through
\begin{align*}
&  \log\int_{C}\exp\left(  -\sum_{1\leq i\leq n}\left(  g(x_{i})+q(x_{i}%
)\right)  \right)  dx_{1},...,dx_{n}\\
&  =-I_{g,q}(A)+\log\int_{0}^{\infty}Volume(A\cap S_{g,q}(I_{g,q}%
(A)+s))e^{-s}ds\\
&  \leq-I_{g,q}(A)+\log\int_{0}^{\infty}\left(  I_{g,q}(A)+s\right)
^{n}e^{-s}ds+n\log2.
\end{align*}
According to $(\ref{th191})$, it holds
\[
\int_{0}^{\infty}\left(  I_{g,q}(A)+s\right)  ^{n}e^{-s}ds\leq(n+1)I_{g,q}%
(A)^{n},
\]
Hence we have
\begin{align*}
&  \log\int_{A}\exp\left(  -\sum_{1\leq i\leq n}\left(  g(x_{i})+q(x_{i}%
)\right)  \right)  dx_{1},...,dx_{n}\\
&  \leq-I_{g,q}(A)+\log\left(  (n+1)I_{g,q}(A)^{n}\right)  +n\log2\\
&  =-I_{g,q}(A)+n\log I_{g,q}(A)+\log(n+1)+n\log2.
\end{align*}
Replacing $A$ by $I^{c}\cap C$ yields $(\ref{th13})$.

\bigskip

\textbf{Step 4}: In this step, we derive crude bounds for $I_{g_{2}%
}(C),I_{g_{1}}(I^{c}\cap C)$ and $I_{g_{2}}(I^{c}\cap C)$.

From $(\ref{th1011})$ and $(\ref{th430})$, there exists some $a_{n}\in\lbrack
X,\infty)$ (say, $a_{n}>y_{2}$) such that
\begin{equation}
M(x)\leq\max(N\log g(a_{n}),N\log g(x)) \label{the20}%
\end{equation}
holds on $(0,\infty).$ Hence for $a_{n}$ large enough
\[
g_{2}(x)=g(x)+M(x)\leq g(x)+\max(N\log g(a_{n}),N\log g(x)),
\]
which in turn yields
\begin{equation}
I_{g_{2}}(C)\leq\inf_{\mathbf{x}\in C}\left(  \sum_{i=1}^{n}g(x_{i}%
)+\sum_{i=1}^{n}\max(N\log g(a_{n}),N\log g(x_{i}))\right)  . \label{the230}%
\end{equation}
It holds
\begin{equation}
\inf_{\mathbf{x}\in C}\left(  \sum_{i=1}^{n}\max(N\log g(a_{n}),N\log
g(x_{i}))\right)  =nN\log g(a_{n}) \label{the231}%
\end{equation}
which implies that
\begin{align*}
&  \inf_{\mathbf{x}\in C}\left(  \sum_{i=1}^{n}g(x_{i})+\sum_{i=1}^{n}%
\max(N\log g(a_{n}),N\log g(x_{i}))\right) \\
&  =\inf_{\mathbf{x}\in C}\left(  \sum_{i=1}^{n}g(x_{i})\right)
+\inf_{\mathbf{x}\in C}\left(  \sum_{i=1}^{n}\max(N\log g(a_{n}),N\log
g(x_{i}))\right) \\
&  =\inf_{\mathbf{x}\in C}\left(  \sum_{i=1}^{n}g(x_{i})\right)  +nN\log
g(a_{n})\\
&  =I_{g}(C)+nN\log g(a_{n})=ng(a_{n})+nN\log g(a_{n}).
\end{align*}
Thus we obtain the inequality
\begin{equation}
I_{g_{2}}(C)\leq ng(a_{n})+nN\log g(a_{n}). \label{the31}%
\end{equation}

We now provide a lower bound of $I_{g_{1}}(I^{c}\cap C)$. Consider the
inequality of $(\ref{th42})$ in \textbf{Step 1}, where we have showed that $h
$ is convex for $x$ large enough; hence, using $(\ref{th42})$ when $a_{n}$ is
sufficiently large, it holds%

\[
I_{g_{1}}(I^{c}\cap C)\geq I_{h}(I^{c}\cap C)=\min\left(  F_{h_{1}}%
(a_{n},\epsilon_{n}),F_{h_{2}}(a_{n},\epsilon_{n})\right)  ,
\]
where the second inequality holds from Lemma $\ref{Lemma01}$. By the
definition of the function $h$ in $(\ref{th4100})$, for large $x$ it holds
$h(x)=r(x)$ which yields the following lower bound of $I_{g_{1}}(I^{c}\cap C)$%

\[
I_{g_{1}}(I^{c}\cap C)\geq I_{h}(I^{c}\cap C)=I_{r}(I^{c}\cap C)=\min\left(
F_{r_{1}}(a_{n},\epsilon_{n}),F_{r_{2}}(a_{n},\epsilon_{n})\right)  .
\]
By Lemma $\ref{Lemma01}$, it holds
\begin{align*}
F_{r_{1}}(a_{n},\epsilon_{n})  &  =g(a_{n}+\epsilon_{n})+(n-1)g\left(
a_{n}-\frac{1}{n-1}\epsilon_{n}\right) \\
&  \qquad\qquad-N\log g(a_{n}+\epsilon_{n})-(n-1)N\log g\left(  a_{n}-\frac
{1}{n-1}\epsilon_{n}\right) \\
&  \geq g(a_{n}+\epsilon_{n})+(n-1)g\left(  a_{n}-\frac{1}{n-1}\epsilon
_{n}\right)  -nN\log g\left(  a_{n}+\epsilon_{n}\right)  ,
\end{align*}
by the same way, we have also
\[
F_{r_{2}}(a_{n},\epsilon_{n})\geq g(a_{n}-\epsilon_{n})+(n-1)g\left(
a_{n}+\frac{1}{n-1}\epsilon_{n}\right)  -nN\log g\left(  a_{n}+\epsilon
_{n}\right)  ,
\]
hence
\[
I_{g_{1}}(I^{c}\cap C)\geq\min\left(  F_{g_{1}}(a_{n},\epsilon_{n}),F_{g_{2}%
}(a_{n},\epsilon_{n})\right)  -nN\log g\left(  a_{n}+\epsilon_{n}\right)
\]
holds.

The method of the estimation of the upper bound of $I_{g_{1}}(I^{c}\cap C)$ is
similar to that used for $I_{g_{1}}(C)$ above. In $(\ref{the230})$, replace
$C$ by $I^{c}\cap C;$ we obtain%

\begin{align*}
I_{g_{2}}(I^{c}\cap C)  &  \leq\inf_{\mathbf{x}\in I^{c}\cap C}\left(
\sum_{i=1}^{n}g(x_{i})+\sum_{i=1}^{n}\max(N\log g(a_{n}),N\log g(x_{i}%
))\right) \\
&  \leq\inf_{\mathbf{x}\in I^{c}\cap C}\left(  \sum_{i=1}^{n}g(x_{i}%
)+\sum_{i=1}^{n}\max(N\log g\left(  a_{n}+\frac{\epsilon_{n}}{n-1}\right)
,N\log g(x_{i}))\right)  .
\end{align*}
Similarly to $(\ref{the231}),$ \ it holds
\[
\inf_{\mathbf{x}\in I^{c}\cap C}\left(  \sum_{i=1}^{n}\max(N\log g\left(
a_{n}+\frac{\epsilon_{n}}{n-1}\right)  ,N\log g(x_{i}))\right)  =nN\log
g\left(  a_{n}+\frac{\epsilon_{n}}{n-1}\right)  ,
\]
where equality is attained setting $x_{1}=...=x_{n-1}=a_{n}+\epsilon
_{n}/(n-1),x_{n}=a_{n}-\epsilon_{n}$. Hence we have, when $n\rightarrow
\infty$
\begin{align*}
I_{g_{2}}(I^{c}\cap C)  &  \leq\inf_{\mathbf{x}\in I^{c}\cap C}\left(
\sum_{i=1}^{n}g(x_{i})+\sum_{i=1}^{n}\max(N\log g\left(  a_{n}+\frac
{\epsilon_{n}}{n-1}\right)  ,N\log g(x_{i}))\right) \\
&  =\inf_{\mathbf{x}\in I^{c}\cap C}\sum_{i=1}^{n}g(x_{i})+nN\log g\left(
a_{n}+\frac{\epsilon_{n}}{n-1}\right) \\
&  =I_{g}(I^{c}\cap C)+nN\log g\left(  a_{n}+\frac{\epsilon_{n}}{n-1}\right)
\\
&  \leq g(a_{n}-\epsilon_{n})+(n-1)g\left(  a_{n}+\frac{1}{n-1}\epsilon
_{n}\right)  +nN\log g\left(  a_{n}+\frac{\epsilon_{n}}{n-1}\right) \\
&  \leq ng\left(  a_{n}+\frac{\epsilon_{n}}{n-1}\right)  +nN\log g\left(
a_{n}+\frac{\epsilon_{n}}{n-1}\right) \\
&  \leq n(N+1)g\left(  a_{n}+\frac{\epsilon_{n}}{n-1}\right)  .
\end{align*}
Therefore we obtain
\begin{equation}
\log I_{g_{2}}(I^{c}\cap C)\leq\log n+\log(N+1)+\log g\left(  a_{n}%
+\frac{\epsilon_{n}}{n-1}\right)  . \label{the41}%
\end{equation}

\textbf{Step 5:} In this step, we complete the proof by showing that
\[
\lim_{a_{n}\rightarrow\infty}\frac{P(I^{c}\cap C)}{P(C)}=0.
\]

Using the upper bound of $P(I^{c}\cap C)$, together with the lower bound of
$P(C)$ above, we have under condition $(\ref{thcond12})$ when $a_{n}$ is large
enough
\begin{align*}
\frac{P(I^{c}\cap C)}{P(C)}  &  \leq\exp\left(
\begin{array}
[c]{c}%
-\left(  I_{g,q}(I^{c}\cap C)-I_{g,q}(C)\right)  +n\log I_{g,q}(I^{c}\cap C)\\
+\tau_{n}+n\log g(a_{n})+\log(n+1)+n\log2
\end{array}
\right) \\
&  \leq\exp\left(  -\left(  I_{g,q}(I^{c}\cap C)-I_{g,q}(C)\right)  +n\log
I_{g,q}(I^{c}\cap C)+\tau_{n}+2n\log g(a_{n})\right) \\
&  \leq\exp\left(  -\left(  I_{g_{1}}(I^{c}\cap C)-I_{g_{2}}(C)\right)  +n\log
I_{g_{2}}(I^{c}\cap C)+\tau_{n}+2n\log g(a_{n})\right)  .
\end{align*}
The last inequality holds from $(\ref{th100})$ and $(\ref{th101})$. Replace
$I_{g_{1}}(I^{c}\cap C),I_{g_{2}}(C)$ by the upper bound of $(\ref{the31})$
and the lower bound of $(\ref{theoremnonconvex12})$, respectively, we obtain
\begin{align}
I_{g_{1}}(I^{c}\cap C)-I_{g_{2}}(C)  &  \geq\min\left(  F_{g_{1}}%
(a_{n},\epsilon_{n}),F_{g_{2}}(a_{n},\epsilon_{n})\right)  -nN\log g\left(
a_{n}+\epsilon_{n}\right) \nonumber\label{2theo31}\\
&  \qquad\qquad-\left(  ng(a_{n})+nN\log g(a_{n})\right) \nonumber\\
&  =H(a_{n},\epsilon_{n})-nN\log g\left(  a_{n}+\epsilon_{n}\right)  -nN\log
g(a_{n})\nonumber\\
&  \geq H(a_{n},\epsilon_{n})-2nN\log\left(  a_{n}+\epsilon_{n}\right)  .
\end{align}

Under condition $(\ref{thcond12})$, there exists some $Q$ such that $n\log
n\leq Qn\log g(a_{n})$, which, together with $(\ref{the41})$ and
$(\ref{2theo31})$, gives%

\begin{align}
\label{theoremnonconvex14}\frac{P(I^{c}\cap C)}{P(C)}  &  \leq\exp\left(
\begin{array}
[c]{c}%
-\left(  H(a_{n},\epsilon_{n})-2nN\log\left(  a_{n}+\epsilon_{n}\right)
\right)  +n\log n+n\log(N+1)\nonumber\\
+n\log g\left(  a_{n}+\frac{\epsilon_{n}}{n-1}\right)  +\tau_{n}+2n\log
g(a_{n})
\end{array}
\right) \nonumber\\
&  =\exp\left(
\begin{array}
[c]{c}%
-H(a_{n},\epsilon_{n})+n(2N+1)\log g\left(  a_{n}+\epsilon_{n}\right) \\
+\tau_{n}+2n\log g(a_{n})+n\log n+n\log(N+1)
\end{array}
\right) \nonumber\\
&  \leq\exp\left(  -H(a_{n},\epsilon_{n})+n(2N+1)\log g\left(  a_{n}%
+\epsilon_{n}\right)  +\tau_{n}+2n\log g(a_{n})+2n\log n\right) \nonumber\\
&  \leq\exp\left(  -H(a_{n},\epsilon_{n})+n(2N+1)\log g\left(  a_{n}%
+\epsilon_{n}\right)  +\tau_{n}+(2Q+2)n\log g(a_{n})\right) \nonumber\\
&  \leq\exp\left(  -H(a_{n},\epsilon_{n})+n(2N+2Q+3)\log g\left(
a_{n}+\epsilon_{n}\right)  +\tau_{n}\right)  .
\end{align}

The second term in the bracket \ in the last line above and $\tau_{n}$ are
both of small order with respect to $H(a_{n},\epsilon_{n})$. Indeed under
condition $(\ref{thcond13})$, when $a_{n}\rightarrow\infty$, it holds
\begin{align}
\label{theoremnonconvex15}\lim_{n\rightarrow\infty}\frac{n(2N+2Q+3)\log
g\left(  a_{n}+\frac{\epsilon_{n}}{n-1}\right)  }{H(a_{n},\epsilon_{n})}=0.
\end{align}

For $\tau_{n}$ which is defined in $(\ref{th51})$under conditions
$(\ref{thcond13}),(\ref{thcond14})$, $nN\log g(a_{n})$ and $nG(a_{n})$ are
both of smaller order than $H(a_{n},\epsilon_{n})$. As regards to the third
term of $\tau_{n}$, it holds
\begin{align*}
nN\log g\left(  a_{n}+\frac{1}{g(a_{n})}\right)   &  =nN\log\left(  g\left(
a_{n}+\frac{1}{g(a_{n})}\right)  -g(a_{n})+g(a_{n})\right) \\
&  \leq nN\log\left(  2\max\left(  G(a_{n}),g(a_{n})\right)  \right) \\
&  =nN\log2+\max\left(  nN\log G(a_{n}),nN\log g(a_{n})\right)  .
\end{align*}
Under conditions $(\ref{thcond13})$ and $(\ref{thcond14})$, \ both $nN\log
G(a_{n})$ and $nN\log g(a_{n})$ are small with respect to $H(a_{n}%
,\epsilon_{n});$ therefore $nN\log g\left(  a_{n}+{1}/{g(a_{n})}\right)  $ is
small with respect to $H(a_{n},\epsilon_{n})$ when $a_{n}\rightarrow\infty$.
Hence it holds when $a_{n}\rightarrow\infty$
\[
\lim_{n\rightarrow\infty}\frac{\tau_{n}}{H(a_{n},\epsilon_{n})}=0.
\]
Finally, $(\ref{theoremnonconvex14})$, together with
$(\ref{theoremnonconvex15})$ and $(\ref{theoremnonconvex16})$, implies that
$(\ref{theoremnonconvex13})$ holds.

\bigskip

\end{document}